\documentclass[a4paper,11pt]{article}
\usepackage{mymacros}
\usepackage{graphics} 
\usepackage{epsfig} 
\usepackage{amsmath} 
\usepackage{amssymb}  
\usepackage{amsthm}
\usepackage{algorithm,algorithmic}
\usepackage{cite}%

\usepackage{authblk} 

\usepackage{tikz,pgfplots} 
\pgfplotsset{compat=newest} 
\pgfplotsset{plot coordinates/math parser=false} 
\usepackage{color}

\newtheorem{theo}{Theorem}
\newtheorem{prop}{Proposition}

\newtheorem{cor}{Corollary}
\newtheorem{mydef}{Definition}
\newtheorem{assump}{Assumption}

\newcommand{\ds}{\displaystyle}
\newcommand{\Pg}{\mathcal{P}}

\newcommand{\Qg}{\mathcal{Q}}
\newcommand{\Rs}{\mathcal{R}}
\newcommand{\Os}{\mathcal{O}}

\newcommand{\sspan}[1]{\ensuremath{\mathop{\mathrm{span}}\left( #1 \right)}}

\usepackage{float}
\usepackage{color}
\usepackage{array}
\usepackage{dsfont}
\usepackage{tikz}
\usetikzlibrary{arrows,decorations.pathmorphing,backgrounds,fit,positioning,shapes.symbols,chains}

\date{}
\title{New Gramians for Linear Switched Systems: Reachability, Observability, and Model Reduction}

\author[1]{Igor Pontes Duff}
\author[1]{Sara Grundel}
\author[1,2]{Peter Benner}

\affil[1]{Max Planck Institute for Dynamics of Complex Technical Systems, Sandtorstr. 1, 39106 Magdeburg, email: \{pontes,~grundel,~benner\}@mpi- magdeburg.mpg.de.}
\affil[2]{Technische Universit\"at Chemnitz, Faculty of Mathematics, Reichenhainer Straße 41, 09126 Chemnitz, Germany}
\begin{document}

\maketitle
\thispagestyle{empty}
\pagestyle{empty}

\begin{abstract}
In this paper, we propose new algebraic Gramians for continuous-time linear switched systems, which satisfy generalized Lyapunov equations. The main contribution of this work is twofold. First, we show that the ranges of those Gramians encode the reachability and observability spaces of a linear switched  system. As a consequence, a simple Gramian-based criterion for reachability and observability is established. Second, a balancing-based model order reduction technique is proposed and, under some sufficient conditions,  stability preservation and an error bound are shown. Finally,  the efficiency of the proposed method is illustrated by means of numerical examples.
\end{abstract}


\section{Introduction}
 
 We consider a continuous-time linear switched system (see \cite{sun2006switched, liberzon2012switching}) (abbreviated by LSS) given by  \begin{equation}\label{eq:LSS}
\Sigma_{LSS} : \left\{\begin{array}{l}
\dot{x}(t) = A_{q(t)}x(t)+B_{q(t)}u(t), \hfill x(0) = x_0,\\ y(t)=C_{q(t)}x(t), 
\end{array}\right. 
\end{equation}
where $\Omega = \{1, \dots, M\}$  is the set of different modes of $\Sigma_{LSS}$, $x(t) \in \R^n$ is the state, $u(t) \in \R^m$ is the controlled input, $y(t)\in \R^{p}$ is the measured output and $q(t)$ is the switching signal, \emph{i.e.}, a piecewise constant function taking values from the index set $\Omega$. The system matrices $A_j \in \R^{n \times n}$, $B_j \in \R^{n \times m}$ and $C_j\in\R^{p \times n}$, where $j \in \Omega$, correspond to the linear system active in mode $q$, and $x_0$ is the initial state.   Furthermore, let $x(t) = \phi(t,x_0,u,q)$  denote the state trajectory at time $t$ of an LSS  initialized at $x(0) = x_0\in \R^n$, with input $u$ and switching signal $q$. In what follows, we assume zero initial condition, \emph{i.e.},  $x(0) =0$ in \eqref{eq:LSS} and, for $j \in \Omega$, the matrices $A_j$ are Hurwitz. 
 When these models are of large-scale, modern analysis, simulation and optimization tools become drastically inefficient and thus, model order reduction (MOR) may become necessary. 

In the context of linear time-invariant systems, several model reduction approaches  have been efficiently developed since the 1960s (see the monograph \cite{morAnt05} and the recent surveys \cite{morBauBF14,morAntBG10}). However, reliable MOR techniques  for switched systems have been only studied in recent years. For discrete-time  linear switched systems see for instance \cite{bacstuug2016reachability}  for reachability and observability reduction with constrained switching,  \cite{zhang2008h, birouche2011model, birouche2012model} for $\mathcal{H}_{\infty}$-type reduction, and \cite{shaker2012model, gao2006model, birouche2010gramian} for balancing-based methods. For continuous-time  linear switched systems, see \cite{bastug2014model, bacstuug2016model, gosea2017data} for a class of moment matching methods, \cite{shaker2011generalised,monshizadeh2012simultaneous, petreczky2012theoretical} for  balancing-based methods and \cite{schulze2018model} for model reduction of systems affected by a low-rank switching. Also, \cite{petreczky2013balanced} presents a theoretical analysis of the techniques proposed in \cite{shaker2011generalised} and \cite{shaker2012model} for continuous- and discrete-time LSS. 

Besides \cite{monshizadeh2012simultaneous}, all of the balancing-based methods rely on Gramains satisfying Linear Matrix Inequalities (LMIs). Although LMIs provide a very flexible tool in control theory, they are costly to solve numerically  in the large-scale setting. To overcome this,  the current paper aims at providing new algebraic Gramians for LSS, denoted by $\Pg$ and $\Qg$, respectively, which satisfy generalized Lyapunov equations. These Gramians are inspired by bilinear model reduction techniques, in which the generalized Lyapunov equation plays an important role, \emph{e.g.}, (see\cite{morBenD11}).  
In addition, we prove that $\Pg$ and $\Qg$ encode the reachability and observability spaces of an LSS, and their kernels correspond to the uncontrollable and unobservable spaces.  

Once the proposed Gramians are computed, by means of a square root balancing approach (see \cite{morAnt05}) and for a given state space dimension $r \ll n$,  we are able to construct two projection matrices $V,W \in \R^{n \times r}$ such that $W^T V = I_r$, which allows us to determine the reduced order LSS as
\begin{equation}\label{eq:LSSred}
\hat{\Sigma}_{LSS} : \left\{\begin{array}{l}
\dot{ \hat {x}}(t) = \hat A_{q(t)}\hat x(t)+\hat B_{q(t)}u(t),  \hfill \hat{x}(0) = 0,\\ \hat y(t)=\hat C_{q(t)} \hat x(t), 
\end{array}\right.
\end{equation}
where 
\begin{equation}
\hat{A}_j = W^TA_jV ,~\,~   \hat B_j = W^TB_j,~\,~\hat C_j= C_jV
\end{equation}
for $l \in \Omega$.   We call $W,V$ global projection matrices $V,W \in \R^{n \times r}$ because they are fixed for every mode $j \in \Omega$ (see \cite{morBenGW15} for a discussion of local and global projection techniques). Readers should refer to \cite{Gosea2017} for a balancing-type method where the projection matrices $V_j,W_j \in \R^{n \times r}$ might depend on the mode $j \in \Omega$, where the authors consider a more general realization then \eqref{eq:LSS}.

\newpage
\noindent\textbf{Outline.} The remaining parts of the paper are organized as follows. Section 2 an LSS is formulated as a bilinear system. Inspired by this transformation, Gramians for LSS are proposed. In Section 3, we prove that those Gramians encode the reachability and observability spaces. Also, we propose a Gramian-based criterion to determine if an LSS is reachable and observable.  In Section 4, the balanced truncation procedure based on these Gramians is introduced. Moreover,  under certain assumption (see \cite{petreczky2013balanced}), this procedure is shown to preserve quadratic stability and to have an error bound. Finally, numerical results are shown in Section 5,  and Section 6 concludes the paper.

\noindent\textbf{Notations.} We denote by $\mathbb{N}$ the set of natural numbers including 0. Let $A \in \R^{n \times m}$ and $B \in \R^{p \times q}$ be two real matrices, we denote $A\otimes B \in \R^{np \times mq}$ the corresponding Kronecker product between $A$ and $B$. 

\section{Bilinear formulation of an LSS and generalized Gramians}

\subsection{Bilinear realization}
In this section, we rewrite the equations of an LSS  to resemble a bilinear system. A very similar procedure was developed in \cite{morBenB11} in the context of parametric systems. 
 
To this aim, first, let us define the matrices
\begin{equation}\label{eq:BilinearMatrices} 
A = A_1~\,~D_j = A_j - A_1, ~\textnormal{for}~j=1,\dots, M.
\end{equation}
Notice, even if $D_1 = 0$, for simplicity, we are going to keep it in the equation. Now, let us replace the switching signal $q(t)$, which takes values in the mode set $\Omega$, by  $M$  switching indicators $\{q_1(t),\dots, q_M(t)\}$, taking binary values, \emph{i.e.}, $q_j(t)\in \{0,1\}$ such that $\ds\sum_{k=1}^M q_k(t) = 1$.  Therefore, 
\[ q(t) = k \Leftrightarrow q_k(t) = 1~ \textnormal{and}~q_j(t) =0~\textnormal{for}~j \neq k.\]
With this notion, the mode $k$ is active when $q_k(t)=1$ and $q_j = 0$ for $j \neq k$. The LSS from \eqref{eq:LSS} can be expressed as

\begin{equation}\label{eq:LSSbilinear}
\begin{array}{l}
\dot{x}(t) = Ax(t)+ \ds\sum_{j=1}^M q_j(t)D_jx(t) + q_j(t)B_{j}u(t), \\
y(t) = \ds\sum_{j=1}^Mq_j(t)C_jx(t).
\end{array}
\end{equation}
Let us include the switching indicators as additional inputs, \emph{i.e.},
\[ (\tilde{u}(t))^T = \begin{bmatrix}
u(t)^T &q_1(t) & \dots & q_M(t)
\end{bmatrix} \in \R^{1 \times (m+M)} \] and 
$\tilde{B}_j = \begin{bmatrix}
B_j & 0
\end{bmatrix} $ for $j \in \Omega$. Then,

\begin{equation}\label{eq:LSSbilinear2}
\begin{array}{l}
\dot{x}(t) = Ax(t)+ \ds\sum_{j=1}^{M} \tilde{u}_{j+m}(t)D_jx(t) + \tilde{u}_{j+m}(t) \tilde{B}_j \tilde{u}(t), \\
y(t) = \ds\sum_{j=1}^{M}\tilde{u}_{j+m}(t)C_jx(t)
\end{array}.
\end{equation}

The crucial observation is that the equations above are very similar to a bilinear system realization, which is usually given as 
\begin{equation}\label{eq:Bilinear}
\begin{array}{l}
\dot{x}(t) = Ax(t)+ \ds\sum_{j=1}^M u_j(t)N_jx(t) + Bu(t) \\ 
y(t) = Cx(t)
\end{array}.
\end{equation}
Hence, if  $B_j = B$ and $C_j = C$ for $j\in \Omega$,  then the realization of an LSS can be recast as a bilinear system. However, in the general case $B_j \neq B_k$ and $C_j \neq C_k$  for $j \neq k$.

In what follows,  we recall some results of model reduction of bilinear systems and, inspired by that, new Gramians for LSS are proposed.

\subsection{Generalized Gramians for LSS}

In the past  years, model reduction of bilinear systems has been studied in the literature, see \cite{morBenD11} for more details. 
A bilinear system as \eqref{eq:Bilinear} is associated to the reachability and observability Gramians
 \begin{subequations}\label{eq:BGram}
	\begin{align}
	&\Pg_{B} =\ds \sum_{k=1}^{\infty} \int_0^{\infty}\dots\int_0^{\infty} P_k(t_1,\dots ,t_k)P_k(t_1,\dots ,t_k)^Tdt_1 \dots dt_k,\label{eq:BGramReach} \\
	&\Qg_B =\ds \sum_{k=1}^{\infty} \int_0^{\infty} \dots \int_0^{\infty} Q_k(t_1,\dots ,t_k)Q_k(t_1,\dots ,t_k)^Tdt_1 \dots dt_k, \label{eq:BGramObserv} 
	\end{align}
\end{subequations}
respectively, where
\begin{align*}
&P_1(t_1) = e^{ At_1}  B,   \\
&Q_1(t_1) = e^{A^Tt_1} C^T, \hspace{3cm}  \\ 
&P_k(t_1, \dots, t_k) = e^{At_k}\begin{bmatrix}
N_1P_{k-1} & \dots &  N_M P_{k-1} \end{bmatrix},  \\
&Q_k(t_1, \dots, t_k) = e^{A^Tt_k}\begin{bmatrix}
N_1^T Q_{k-1} & \dots &  N_M^T Q_{k-1}
\end{bmatrix} .
\end{align*}

Moreover,  if the Gramians exist, \emph{i.e.} the infinite sums converge, they satisfy the following generalized Lyapunov equations

\begin{subequations}\label{eq:GenLyapBilinear}
\begin{align}
A\Pg_B + \Pg_B A^T + \ds\sum_{j=1}^{M} \bigg(N_j\Pg_B N_j^T\bigg) + BB^T &= 0, \\
A^T\Qg_B + \Qg_B A + \ds\sum_{j=1}^{M} \bigg(N_j^T\Qg_B N_j\bigg) + CC^T &= 0.
\end{align}
\end{subequations}
Those equations where proposed in \cite{hsu1983realization} and used to construct minimal realizations and model reduction techniques based on balanced truncation of bilinear systems, see \emph{e.g.}  \cite{al1993new,al1994new} and \cite{morZhaL02}. As mentioned before, the realization of an LSS is not equivalent to a bilinear realization because $B_k \neq B_j$ and $C_k \neq C_j$ for $j \neq k$. However, inspired by those expressions, we propose the following Gramians to be associated to a given LSS.
\begin{mydef}[Generalized Gramians for LSS]\label{def:Ggram} Given an LSS as in \eqref{eq:LSS} and the matrices $D_j$ defined in \eqref{eq:BilinearMatrices}. Then let $\Pg, \Qg$ be
\begin{subequations}\label{eq:LSGram}
	\begin{align}
	&\Pg =\ds \sum_{k=1}^{\infty} \int_0^{\infty}\dots\int_0^{\infty} P_k(t_1,\dots ,t_k)P_k(t_1,\dots ,t_k)^Tdt_1 \dots dt_k,\label{eq:BGramReach} \\
	&\Qg =\ds \sum_{k=1}^{\infty} \int_0^{\infty} \dots \int_0^{\infty} Q_k(t_1,\dots ,t_k)Q_k(t_1,\dots ,t_k)^Tdt_1 \dots dt_k, \label{eq:BGramObserv} 
	\end{align}
\end{subequations}	
	where
	\begin{align*}
	&P_1(t_1) = e^{ At_1} \begin{bmatrix}
	B_1 & \dots & B_M
	\end{bmatrix},    \\
	&Q_1(t_1) = e^{A^Tt_1} \begin{bmatrix} C^T_1 & \dots & C^T_M
	\end{bmatrix},  \hspace{3cm}  \\ 
	&P_k(t_1, \dots, t_k) = e^{At_k}\begin{bmatrix}
	D_1P_{k-1} & \dots &  D_M P_{k-1} \end{bmatrix},  \\
	&Q_k(t_1, \dots, t_k) = e^{A^Tt_k}\begin{bmatrix}
	D_1^T Q_{k-1} & \dots &  D_M^T Q_{k-1}
	\end{bmatrix}.
	\end{align*}
If they exist, $\Pg$ and $\Qg$ will be called the reachabillity and observability Gramians of the LSS.
\end{mydef}	
As a consequence, if $\Pg, \Qg $ exist, they are symmetric, positive semidefinite matrices  which satisfy the following generalized Lyapunov equations 
\begin{subequations}
	\begin{align}
A\Pg + \Pg A^T + \ds\sum_{j=1}^{M} \bigg(D_j\Pg D_j^T+  B_jB_j^T \bigg) &= 0, \label{eq:GenLyapLSSreach} \\
A^T\Qg + \Qg A + \ds\sum_{j=1}^{M} \bigg(D_j^T\Qg D_j +C_j^TC_j\bigg) &= 0.\label{eq:GenLyapLSSobsev}
\end{align}
\end{subequations}	


Note that the name "Gamians" will be justified in the following. The LSS Gramians can be computed using the Kronecker product, \emph{i.e.}, let 
\[\mathcal{M} = \bigg(A\otimes I_n + I_n\otimes A +\sum_{j=1}^M D_j \otimes D_j \bigg)\in \R^{n^2 \times n^2},  \]
\[ \mathcal{B} = \vecop{\sum_{k=1}^{M} B_j B_j^T } ~\,~\textnormal{and}~\,~\mathcal{C} =\vecop{\sum_{j=1}^{M} C_j^TC_j }. \]
Then, the generalized reachability and  observability Gramians are given by
\[ \vecop{\Pg} = -\mathcal{M}^{-1}\mathcal{B}~\,~\textnormal{and}~\,~ \vecop{\Qg}  = -\mathcal{M}^{-T} \mathcal{C}.  \]


 However, in this Kronecker form, the solution of the generalized Lyapunov equation is determined by solving as a set of $n(n + 1)/2$ equations in $n(n + 1)/2$ variables, whose cost is  $O(n^6)$ operations. Fortunately,  new efficient methodologies have been developed recently to determine low-rank solutions of these generalized Lyapunov equations (see \cite{damm2008direct},   \cite{benner2013low},  \cite{shank2016efficient} and \cite{jarlebring2017krylov}) which are suitable in the large-scale setting. 

The following theorem, from \cite{morZhaL02}, states a sufficient condition for  existence and uniqueness of $\Pg$ and $\Qg$.


\begin{theo}[Sufficient conditions for existence and uniqueness \cite{morZhaL02}, Theorem 2]\label{theo:GenLyap} Let $A$, $D_j$, $B_j$ and $C_j$ given by the notation about. In addition, suppose that $A$ is Hurwitz.  Then, there exist real scalars $\beta>0$ and $0<\alpha \leq -\max_i(Re(\lambda_i(A)))$ such that 
	\[\|e^{At}\| \leq \beta e^{-\alpha t}. \]
Then, the reachabillity and observability Gramians satisfying \eqref{eq:GenLyapLSSreach} and \eqref{eq:GenLyapLSSobsev}  exist if 
\[  \left\| \sum_{j=1}^M D_jD_j^T\right\|  < \dfrac{2\alpha}{\beta^2}.\]
\end{theo}


Furthermore, under the conditions of Theorem \ref{theo:GenLyap}, the  symmetric positive semidefinite solutions $\Pg$ and $\Qg$ of equation \eqref{eq:GenLyapLSSreach} can be expressed as an infinite sum of symmetric positive semidefinite matrices $\Pg_k$ and $\Qg_k$ (see \cite{morZhaL02} for more details), \emph{i.e.},
 \[\Pg = \ds \sum_{k=1}^{\infty} \Pg_k~\textnormal{and}~\, \Qg = \ds \sum_{k=1}^{\infty} \Qg_k  \] where
\[ \begin{array}{rcl}
A\Pg_1 + \Pg_1 A^T + \ds\sum_{j=1}^M B_jB_j^T &=& 0, \\ 
A^T\Qg_1 + \Qg_1 A + \ds\sum_{j=1}^M C_j^TC_j &=& 0, 
\end{array}  \]
and
\[ \begin{array}{rcl}
A\Pg_k + \Pg_k A^T + \ds \sum_{j=1}^M D_j\Pg_{k-1} D_j^T &=& 0, \\
 A^T\Qg_k + \Qg_k A + \ds \sum_{j=1}^M D_j^T\Qg_{k-1} D_j &=& 0. 
\end{array} \]

From here on, we assume the  existence and uniqueness of positive semidefinite solutions to \eqref{eq:GenLyapLSSreach} and \eqref{eq:GenLyapLSSobsev} and that the conditions of Theorem \ref{theo:GramCond} hold. In what follows, we show that the proposed Gramians encodes the reachability and observability sets of an LSS. 

\section{Gramians and reachability and observability sets}


As previously mentioned, the main goal of this section is to show  that the LSS Gramians encode the reachability and observability sets of an LSS. 

First of all, let us recall the definition and properties of those sets in the context of LSS. The reader should refer to \cite{sun2002controllability} and \cite{sun2006switched}  for more details. Let us start with the notion of reachability and observability sets.  

\begin{mydef}[Reachable set] A state  $x \in \R^n$ is reachable, if there exist a time instant $t_f$, a switching signal $q : [0,t_f] \rightarrow \Omega $, and an input $u:  [0,t_f] \rightarrow \R^p $, such that  $\phi(t_f,0,u,q) = x$. The \emph{reachable set} of  an LSS is denoted by $\Rs$, that is the set of states which are reachable.
\end{mydef}
\begin{mydef}[Observability set] A state $x$ is said to be unobservable, if for any switching signal $q$, there exists an input $u(t)$ such that
\[C_{q(t)}  \phi(t,x,u,q) = C_{q(t)}  \phi(t,0,u,q), ~\forall t \geq 0.   \]
The unobservable set of an LSS, denoted by $\mathcal{U} \mathcal{O}$, is the set of states which are unobservable. The \emph{observable set} of an LSS,   denoted by $\mathcal{O}$, is defined by $\mathcal{O} = (\mathcal{U} \mathcal{O})^\perp$.
\end{mydef}

In what follows, we recall the algebraic characterization of $\Rs$ and $\Os$ and we state the main result of this paper, \emph{i.e.}, the Gramian version of this result.

\subsection{Characterization of the reachability and observability sets}

The following result,  from \cite{sun2006switched},  describes the reachable and observable sets of an LSS  by algebraic conditions.

\begin{theo}[Algebraic conditions \cite{sun2006switched},Theorem $4.17$]\label{theo:AlgCond} For an LSS as in \eqref{eq:LSS}, the reachable and observable sets $\Rs$ and $\Os$ are linear subspaces of $\R^n$ given by 
	\[ \Rs = \sum_{k=1}^{\infty}\left(\sum_{ \substack{
			  i_0, \dots, i_{k} \in \Omega \\
			j_1, \dots, j_{k}\in \mathbb{N}}}  A_{i_{k}}^{j_{k}}\dots A_{i_1}^{j_1}\range{B_{i_0}}\right),  \]
	and
	\[ \Os = \sum_{k=1}^{\infty}\left(\sum_{\substack{i_0, \dots, i_{k} \in \Omega \\ j_1, \dots, j_{k}\in \mathbb{N}}} (A_{i_{k}}^{j_{k}})^T\dots (A_{i_1}^{j_1})^T \range{C_{i_0}^T}\right). \]
\end{theo}

Theorem \ref{theo:AlgCond} generalizes the well-known reachability and observability criteria for LTI systems. In the context of LTI systems, the reachable set is a linear subspace of $\R^n$ given by $ \Rs = \sum_{k=0}^{\infty} A^k\range{B}$, \emph{i.e.}, all possible combinations of one variable polynomials in $A$ multiplied by $B$. In the context of LSS,  the reachable set is also a linear subspace of $\R^n$ given by all possible combination of $M$-variate polynomials  in $A_1, \dots A_M$ multiplied by $B_k$. Moreover, this subspace can be seen as the smallest subspace of $\R^n$ that contains each $\range{B_i}$ and is invariant under each $A_i$, for $i \in \Omega$. 

In what follows, we state the main result of this paper.

\begin{theo}[Gramian conditions]\label{theo:GramCond} Let $\Pg, \Qg$ be the solutions of the generalized Lyapunov equations \eqref{eq:GenLyapLSSreach} and \eqref{eq:GenLyapLSSobsev}, respectively.  Then, the reachable and observable spaces $\Rs$, $\Os$ are given by 
	\begin{equation*}
			\Rs= \range \Pg ~\,~\textnormal{and}~\,~\Os=\range \Qg.
	\end{equation*}
\end{theo}
\begin{proof}
	The proof of this theorem shows that the range of $\Pg$ and $\Qg$ are given by the algebraic condition of Theorem~\ref{theo:AlgCond}. The complete proof is detailed in Appendix A.
\end{proof}

Theorem \ref{theo:GramCond} states a Gramian-based characterization of the reachable and observable sets. Moreover, the following reachability and observability criteria are corollaries of this result. 

\begin{cor}[Reachability and observability criteria]\label{cor:ReachObservCrit}
	 Given $\Sigma$, an LSS,  and suppose that $\Pg, \Qg $ are the unique solutions of the generalized Lyapunov equations \eqref{eq:GenLyapLSSreach} and \eqref{eq:GenLyapLSSobsev}. Then, 
	\begin{enumerate}
			\item $\Sigma$ is completely reachable if and only if \[\range \Pg = \R^n.\]
			\item $\Sigma$ is completely observable if and only if \[\range \Qg = \R^n.\]
	\end{enumerate}
\end{cor}
\begin{proof}
The LSS is completely reachable (respectively, observable) if and only if $\Rs = \R^n$ (respectively, $\Os = \R^n$). Then the result is a straightforward application of Theorem \ref{theo:GramCond}.  
\end{proof}
Corollary \ref{cor:ReachObservCrit} provides simple criteria for determine if a given LSS is completely reachable and observable. This result is equivalent to verifying if the algebraic conditions given in Theorem \ref{theo:AlgCond} generate the entire space. However, to the best of the authors' knowledge, they have not been presented in this Gramian-based form. 

To sum up, the Gramians $\Pg$ and $\Qg$ proposed in Definition~\ref{def:Ggram} encode the reachable and observable spaces of a given LSS (as stated in Theorem \ref{theo:GramCond}). As a consequence, Corollary \ref{cor:ReachObservCrit} provides a simple way to verify if a given LSS is completely reachable and observable. In the next section, we present  the procedure for model order reduction by balanced truncation using these Gramians.

\section{Model reduction for linear switched systems }
In this section, we state the balancing procedure for model reduction of LSS and we state some sufficient conditions under which this procedure preserves stability, and provide an approximation error bound.

\subsection{Balanced truncation for LSS}

As mentioned before, the Gramians $\Pg$ and $\Qg$ encode the reachable and observable spaces. This can be rewritten as follows :
\begin{enumerate}
	\item If a state $x$ lies in $\ker(\Pg)$, then it is unreachable.
	\item If a state $x$ lies in $\ker(\Qg) $, then it is unobservable.
\end{enumerate}
Hence, the subspace $\ker (\Pg\Qg)$ is not important for the transfer between input and output and might be truncated. This motivates us to use the proposed Gramians to determine the reduced-order models. To guarantee that states which are hard to control and hard to observe  will be truncated simultaneously, we need to find a transformation $T$, leading to a transformed switched system, whose controllability and observability Gramians are equal and diagonal, \emph{i.e.},
\[ T^{-1}\Pg T^{-T} = T^T\Qg T  = \Sigma = \diag{\sigma_1, \dots, \sigma_n}, \] 
with $\sigma_i \leq \sigma_{i+1}$. This balancing transformation exists if and only if $\Pg$ and $\Qg$ are full rank matrices (see Chapter 7 of \cite{morAnt05}). Next, we assume that the matrices of the balanced system are partitioned as
\[A_{j,\mathcal{B}} = \begin{bmatrix}
A_{j}^{11} & A_{j}^{12}  \\ A_{j}^{21} & A_{j}^{22}  
\end{bmatrix},~ B_{j,\mathcal{B}} = \begin{bmatrix}
B_{j}^{1}  \\  B_{j}^{2}  
\end{bmatrix},~C_{j,\mathcal{B}} = \begin{bmatrix}
C_{j}^{1}  &  C_{j}^{2}  
\end{bmatrix}~\textnormal{and}~ \Sigma = \begin{bmatrix}
\Sigma_1 & 0  \\ 0 & \Sigma_2 
\end{bmatrix},  \]
where $\Sigma_1 = \diag{\sigma_1, \dots, \sigma_r}$ and $\Sigma_2 = \diag{\sigma_{r+1}, \dots, \sigma_n}$. In the balancing basis, the truncation step is simply obtained by setting the ROM  to be given by the matrices $\hat{A}_j = A_{j}^{11}, \hat{B}_j = B_{j}^{1} ,~\hat{C}_j = C_{j}^{1}$. Analogous to the linear case, we do not need to compute the balanced transformation explicitly. Instead, one can construct two projection matrices $V$ and $W$ using the Cholesky factors of $\Pg$ and $\Qg$, and the SVD of their product. This procedure is known as square-root balanced truncation, and its version for LSS is presented in Algorithm~1.




\begin{algorithm}\label{algo:BT}
	\caption{Balanced truncation for LSS}
	\begin{algorithmic}[1]
		\renewcommand{\algorithmicrequire}{\textbf{Input:}}
		\renewcommand{\algorithmicensure}{\textbf{Output:}}
		\REQUIRE Matrices $(A_j,B_j,C_j)$ for $j=1,\dots M$ and reduced order $r$. 
		\ENSURE  Reduced order matrices $\Sigma(\hat A_j,\hat B_j,\hat C_j)$ for $j=1,\dots M$.
		\STATE Let $A = A_1$, $D_k = A_k-A_1$.
		\STATE Compute $\Pg$ and $\Qg$ by solving the generalized Lyapunov equations \eqref{eq:GenLyapLSSreach} and \eqref{eq:GenLyapLSSobsev}.
	 
	 \STATE Compute the Cholesky decomposition $\Pg = SS^T$ and $\Qg = RR^T$.
	 \STATE Compute SVD of $S^TR$ and write as \[ S^TR =U\Sigma V^T = \begin{bmatrix}
	 	U_1 & U_2
	 	\end{bmatrix}\diag{\Sigma_1, \Sigma_2}\begin{bmatrix}
	 	V_1 & V_2
	 	\end{bmatrix}^T \]
	 	\STATE Construct the projection matrices $ V = SU_1\Sigma_1^{-\frac{1}{2}}$ and $W = RV_1\Sigma_1^{\frac{1}{2}} $  
	 	\STATE Construct $\hat{A}_j = W^TA_jV$, $\hat{B}_j = W^TB_j$ and $\hat{C}_j = C_jV$ for $j=1, \dots, M$.
		\RETURN $\hat{A}_j, \hat{B}_j$ and $\hat{C}_j$ .
	\end{algorithmic} 
\end{algorithm}

One should notice that, if matrix $A$ is Hurwitz, the proposed procedure provides a matrix $\hat{A} = W^TAV$ which is also Hurwitz. This is a consequence of Theorem 2.3 from \cite{benner2016positive}. In a large-scale setting, a solution of the generalized Lyapunov equation is computed directly in the factorized form, \emph{i.e.},  one searches for the solution $S$ as a low-rank factor such that $\Pg \approx SS^T$  (see \cite{benner2013low},  \cite{shank2016efficient} and \cite{jarlebring2017krylov}). In this context, one can avoid constructing the full solutions $\Pg, \Qg$, which is very costly with respect to memory consumption and computational resources.

In the next subsection, under some assumptions, we show  some properties of the reduced order models obtained by Algorithm 1.     


\subsection{Quadratic stability preservation and error bounds}

We briefly review  the definition of quadratic stability for LSS.

\begin{mydef}[Quadratic stability \cite{petreczky2013balanced}, Lemma 1] An LSS as in \eqref{eq:LSS} is said to be quadratically stable if there exists a positive definite matrix $P>0$ such that
	\[ A_j^TP+PA_j <0,~\textnormal{for all $j\in \Omega$}. \]
\end{mydef}
Quadratic stability is a sufficient condition for exponential stability for all switching signals (see \cite{liberzon2012switching}). In what follows in this section, we employ the following assumption.

\begin{assump}\label{assump:LMI} Let $\Pg$ and $\Qg$ be symmetric positive definite  solutions of \eqref{eq:GenLyapLSSreach} and \eqref{eq:GenLyapLSSobsev}.  Let us assume that
\begin{subequations}
\begin{align}
D_k \Pg + \Pg D_k^T &\leq \sum_{j=1}^M D_j \Pg D_j^T +  \sum_{j=1, j\neq k}^M B_jB_j^T, ~\,~ and \\
D_k^T \Qg + \Qg D_k &\leq \sum_{j=1}^M D_j^T \Qg D_j +  \sum_{j=1, j\neq k}^M C_j^TC_j,
\end{align}
\end{subequations}
for every $k=2,\dots, M$.
\end{assump}
Reader should notice that Assumption~\ref{assump:LMI} implies that
\begin{subequations}\label{eq:LMIGram}
\begin{align}
A_k\Pg +\Pg A_k^T +B_kB^T_k &\leq 0, \\ 
A_k^T\Qg +\Qg A_k +C_k^TC_k &\leq 0,
\end{align}
\end{subequations}
for every $k\in \Omega$. Hence, under Assumption 1, the Gramians proposed in this work are also Gramians in the sense of  \cite{petreczky2013balanced} (Definition 10), \emph{i.e.}, symmetric positive definite matrices which satisfy the set of LMIs \eqref{eq:LMIGram}. As a consequence, under the Assumption \ref{assump:LMI}, all of the results developed in \cite{petreczky2013balanced} are also valid for the Gramians proposed in \eqref{eq:LSGram}. Two are particularly important for model reduction, and we recall them in what follows. 

\begin{prop}[Quadratically stability preservation\cite{petreczky2013balanced}, Lemma 12]\label{prop:QuadStabPreserv}  Under Assumption \ref{assump:LMI}, if at least one of the inequalities \eqref{eq:LMIGram} is strict, \emph{i.e.}, 
	 \[A_k\Pg +\Pg A_k^T +B_kB^T_k < 0, \forall k\in \Omega, ~\,or~\, A_k^T\Qg +\Qg A_k +C_k^TC_k < 0, \forall k\in \Omega, \]
then the reduced order model constructed by Algorithm 1 is also quadratically stable. 
\end{prop}

Proposition \ref{prop:QuadStabPreserv} states a sufficient condition to preserve quadratic stability by model reduction using Algorithm 1.  The following result provides an error bound between the origial and the reduced order model.

\begin{theo}[Error bound\cite{petreczky2013balanced}, Theorem 6]\label{theo:ErrorBound} Under Assumption \ref{assump:LMI}, the output error between the original model and the reduced order model \eqref{eq:LSSred}, obtained by Algorithm 1, is bounded by
	\begin{equation}\label{eq:ErrorBound}
	\|y -\hat{y}\|_{L_2} \leq 2 \left(\sum_{k=r+1}^n \sigma_k\right) \|u\|_{L_2} 
	\end{equation}
for every switching signal $q(t)$, where $\sigma_k$ are the neglected singular values. 
\end{theo} 

Proposition \ref{prop:QuadStabPreserv} and Theorem \ref{theo:ErrorBound} provide some important properties of model reduction by balanced truncation using the Gramians from Definition~\ref{def:Ggram}. However, they are only proved here to be valid when Assumption 1 holds. Since this assumption involves LMIs, they are hard to be checked in the large-scale setting. We left as an open problem  whether weaker assumptions exist such that similar results are also valid.

In the next section, we apply the results derived in this paper in some numerical examples. 
\section{Numerical examples}
This section is dedicated to the application of results proposed in Sections III and  IV, namely the Gramian-based characterization (Theorem \ref{theo:GramCond}) of the reachable and observable spaces and the balanced truncation procedure  (Algorithm 1).  The results will be compared with the balancing method proposed \cite{monshizadeh2012simultaneous}. There, it has been shown that, if certain restrictive conditions are satisfied, a simultaneous balanced transformation can be constructed. When those conditions are not satisfied, the authors propose  to use, instead, the so-called  reachability and observability average Gramians given by
\[ \Pg_{avg} = \sum_{k=1}^M \Pg_k~\,~\textnormal{and}~\,~   \Qg_{avg} = \sum_{k=1}^M \Qg_k, \]
which satisfy
\[ \begin{array}{rcl}
 A_k\Pg_k + \Pg_kA_k^T +B_kB_k^T &=& 0, \\ A_k^T\Qg_k + \Qg_kA_k +C_k^TC_k &=& 0.
\end{array}\]
 
 In what follows, we illustrate the Gramian-based characterization of the reachbility set using Theorem \ref{theo:GramCond}.

\subsection{Example 1: Reachability set of LSS}
 
Let us consider  an 2-modes LSS $\Sigma$ given by
\[A_1 = -I_{8}, A_2 = A_1+D, \] 
where $D\in \R^{8\times 8}$ satisfies $D_{21} = D_{32} = D_{43} = 1$ and $D_{jk} =0$ elsewhere. In addition,
 $B_1^T = \begin{bmatrix}
1 & 0 & \dots &  0
\end{bmatrix} $, $B_2^T = \begin{bmatrix}
0 & \dots & 0 & 1
\end{bmatrix}$. 
Then, the reachability Gramian $\Pg$ given by equation \eqref{eq:GenLyapLSSreach} is
\[\Pg = \diag{\frac{1}{2},\frac{1}{4} , \frac{1}{8}, \frac{1}{16}, 0, 0, 0, \frac{1}{2}} \]
and the average reachability Gramian (proposed in \cite{monshizadeh2012simultaneous}) is 
\[\Pg_{avg} = \diag{\frac{1}{2}, 0, 0, 0, 0, 0, 0, \frac{1}{2}}. \]
As a consequence, since $\range{\Pg} \neq 8$, Corollary \ref{cor:ReachObservCrit} tells us that $\Sigma$ is not completely reachable. In addition, according to Theorem \ref{theo:GramCond},  the reachable space of $\Sigma$ is given by
\[ \Rs = \range{\Pg} = \sspan{e_1, e_2, e_3, e_4, e_8} \]
Notice that the average Gramian $\Pg_{avg}$ does not encode the reachability space. More generally, one can show that 
\[ \range{\Pg_{avg}} \subset \range{\Pg}. \]  

In what follows, we use the Gramians to construct reduced order models via Algorithm~1.
\subsection{Example 2: Model reduction by balancing}

For the next experiment, let us consider a 2-modes LSS of order 1000, whose matrices are given by
\[ A_1 =\ds \begin{bmatrix}
-2 & 1 & &  \\
0.1 & -2 & 1 &   \\
& \ddots &\ddots &\ddots  \\
& &   0.1  & -2 \\
\end{bmatrix},~A_2 = \begin{bmatrix}
-2 & 0.5 & &  \\
1 & -2 & 0.5 &   \\
& \ddots &\ddots &\ddots  \\
& &   0.1  & -2 \\
\end{bmatrix},  \]
$B_1^T =\begin{bmatrix}
1 & 0 & \dots & 0
\end{bmatrix},~B_2^T =\begin{bmatrix}
0 & \dots & 0 & 1
\end{bmatrix},  $
$C_1 =\begin{bmatrix}
0 & 1 & 0 & \dots & 0
\end{bmatrix}$ and $C_2 =\begin{bmatrix}
0 & \dots & 0 & 1 & 0
\end{bmatrix}.$

Using $A_1 = A$ and $D = A_2-A_1$, we compute the generalized Gramians satisfying
\[ \begin{array}{rcl}
A\Pg +\Pg A^T + D\Pg D^T + B_1B_1^T + B_2B_2^T   &=& 0, \\
A^T\Qg +\Qg A + D^T\Qg D + C_1^TC_1 + C_2^TC_2  &=& 0,
\end{array} \] 
and the averaged Gramians $\Pg_{avg} = \frac{1}{2}(\Pg_1 +\Pg_2)$ and $\Qg_{avg} = \frac{1}{2}(\Qg_1 +\Qg_2)$. 
The Hankel singular values are represented in Figure 1. 

 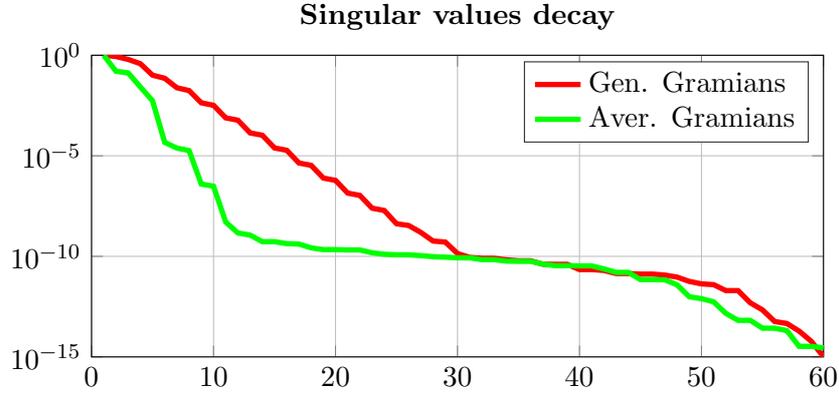
\begin{figure}[H]\label{fig:HankelSVD}
	\newlength\wex
	\newlength\hex
	\setlength{\wex}{0.8\textwidth}
	\setlength{\hex}{4cm}
	\begin{center}
%
%
\begin{tikzpicture}

\begin{axis}[%
width=0.951\wex,
height=\hex,
at={(0\wex,0\hex)},
scale only axis,
xmin=0,
xmax=60,
ymode=log,
ymin=1e-15,
ymax=1,
yminorticks=true,
axis background/.style={fill=white},
title style={font=\bfseries},
title={Singular values decay},
xmajorgrids,
ymajorgrids,
yminorgrids,
legend style={legend cell align=left, align=left, draw=white!15!black}
]
\addplot [color=red, line width=2.0pt]
  table[row sep=crcr]{%
1	1\\
2	0.870386929663933\\
3	0.620967985779413\\
4	0.381046130354309\\
5	0.102300017797286\\
6	0.0715355224743715\\
7	0.0241976181204495\\
8	0.0175167576076616\\
9	0.00429826258815177\\
10	0.00326057980159347\\
11	0.000770888690286747\\
12	0.000590388607772512\\
13	0.000137587304586122\\
14	0.000105582403512743\\
15	2.45593954027926e-05\\
16	1.88535896104863e-05\\
17	4.38356922058738e-06\\
18	3.36539025303425e-06\\
19	7.82407962333454e-07\\
20	6.00692106139554e-07\\
21	1.39615632658389e-07\\
22	1.07217027393209e-07\\
23	2.48407362889515e-08\\
24	1.91363333285628e-08\\
25	4.18925790323968e-09\\
26	3.41726327530598e-09\\
27	1.54099218723512e-09\\
28	5.75918698690828e-10\\
29	5.14958757155974e-10\\
30	1.3812137168487e-10\\
31	8.91013328736501e-11\\
32	8.03494284027979e-11\\
33	8.03494284027979e-11\\
34	6.711875917713e-11\\
35	6.02431602838472e-11\\
36	6.02431602838472e-11\\
37	4.04756489164418e-11\\
38	4.01296304137559e-11\\
39	4.01296304137559e-11\\
40	2.19959779719275e-11\\
41	2.19959779719275e-11\\
42	2.00071031971053e-11\\
43	1.40073776884031e-11\\
44	1.40073776884031e-11\\
45	1.31040161778253e-11\\
46	1.31040161778253e-11\\
47	1.16795878918653e-11\\
48	9.29797301059748e-12\\
49	5.66697562663989e-12\\
50	4.28317753673888e-12\\
51	3.91142609863048e-12\\
52	1.97872165905511e-12\\
53	1.97872165905511e-12\\
54	4.89135806859307e-13\\
55	2.17644487929198e-13\\
56	5.74101921226992e-14\\
57	4.58360203867954e-14\\
58	1.92831529445971e-14\\
59	5.96421345135598e-15\\
60	1.02208646497557e-15\\
};
\addlegendentry{Gen. Gramians}

\addplot [color=green, line width=2.0pt]
  table[row sep=crcr]{%
1	1\\
2	0.160426720319961\\
3	0.132855106540704\\
4	0.0265090133373836\\
5	0.00545565817729584\\
6	4.64294337201604e-05\\
7	2.42829606436996e-05\\
8	1.81213682668579e-05\\
9	3.91657689168934e-07\\
10	3.0886008340163e-07\\
11	5.10890695417003e-09\\
12	1.45182111126207e-09\\
13	1.1417742151148e-09\\
14	5.37153226346617e-10\\
15	5.37153226346617e-10\\
16	4.28181191711959e-10\\
17	4.0769126178212e-10\\
18	2.64725032016027e-10\\
19	2.16112368665511e-10\\
20	2.16112368665511e-10\\
21	2.08422591491333e-10\\
22	2.08422591491333e-10\\
23	1.50510388337928e-10\\
24	1.27068416725124e-10\\
25	1.1959485940933e-10\\
26	1.1959485940933e-10\\
27	1.0858894720729e-10\\
28	9.6212193189738e-11\\
29	9.09416289443542e-11\\
30	8.46165517940659e-11\\
31	8.46165517940659e-11\\
32	6.87393515760867e-11\\
33	6.87393515760867e-11\\
34	5.78371350286347e-11\\
35	5.67468941878876e-11\\
36	5.67468941878876e-11\\
37	4.11273160787977e-11\\
38	3.44827016368728e-11\\
39	3.44827016368728e-11\\
40	3.34959096932959e-11\\
41	3.34959096932959e-11\\
42	2.38199474053759e-11\\
43	1.5974346472953e-11\\
44	1.5974346472953e-11\\
45	6.9706560581846e-12\\
46	6.9706560581846e-12\\
47	6.79465337659743e-12\\
48	3.81347950183361e-12\\
49	9.76716928922124e-13\\
50	7.84604640123346e-13\\
51	5.46950225346181e-13\\
52	1.46061556450253e-13\\
53	6.55996568757892e-14\\
54	6.55996568757892e-14\\
55	2.67158872839874e-14\\
56	2.67158872839874e-14\\
57	2.02778771994864e-14\\
58	3.34113699151129e-15\\
59	3.34113699151129e-15\\
60	2.7920471411484e-15\\
};
\addlegendentry{Aver. Gramians}

\end{axis}
\end{tikzpicture}%
		\caption{Hankel singular values decay corresponding to the Generalized Gramians (red  line) and Averaged Gramians (green line).}
	\end{center}
\end{figure}

Choose the truncation order $r= 15$ for the reduced LSS using both methods.  We compare the time domain response of the original LSS against the ones corresponding to the two reduced models. For this, we use  as the control input, $u(t) = 10 \sin(30t)e^{-t} $ and as the switching signal
\[ q(t) = \left\{\begin{array}{rl}
1,&~t\in[0,0.5]\cup[2,2.5]\cup[4,5]\cup[5.5, 6],   \\  
2,&~t\in[0.5,2]\cup[2.5,4]\cup [5,5.5]. 
\end{array}  \right.\]
The results are represented in Figure 2. The absolute errors are represented in Figure 3.

 \begin{figure}[H]\label{fig:AbsError}
	\setlength{\wex}{0.8\textwidth}
	\setlength{\hex}{4cm}
	\begin{center}
		\input{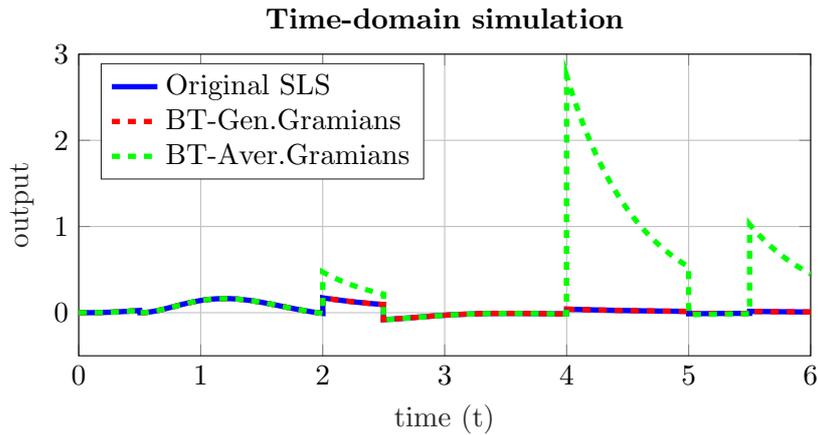}	
		\caption{Output corresponding to the time domain simulation of the original model (blue line), generalized Gramian ROM (red line) and averaged Gramian ROM.}
	\end{center}
\end{figure}

\begin{figure}[H]
	\setlength{\wex}{0.8\textwidth}
	\setlength{\hex}{5cm}
	\begin{center}
	\input{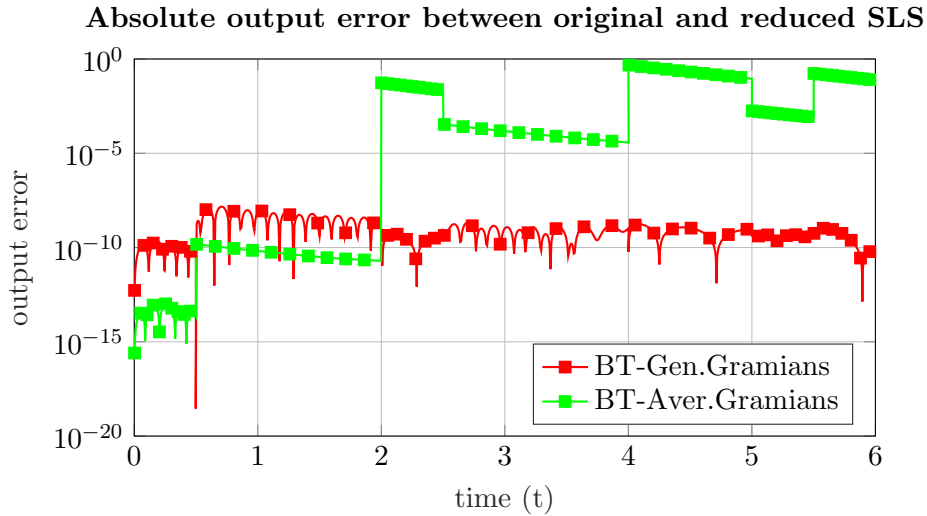}	
	\caption{Output absolute error between the original model and the reduced ones (Generalized Gramians: red line, Averaged Gramians: Green line).}
\end{center}
\end{figure}

By inspecting the time-domain error between the original response and the  two reduced models (Figure~3), we observe that  the new proposed method generally produces better results. In addition, notice that $\|u\|_{L_2} = 10 \sqrt{\frac{225}{901}} \approx 4.997$, so that the error bound of Theorem \ref{theo:ErrorBound} can be computed as $2 \left(\sum_{k=r+1}^n \sigma_k\right) \|u\|_{L_2}  = 5.033\mathrm{e}{-05}$. By numerical computing of the $L_2$-norm of the error between the original and the reduced order model, one obtain $7.00\mathrm{e}{-09}$ for the system obtained using the proposed method and $ 0.2715$ for the one obtained using average Gramians. As a conclusion, the bounds derived in Theorem \ref{theo:ErrorBound} are satisfied for the proposed method.

\section{Conclusion}
\label{section:conclusion}

In this paper, we have proposed new reachability and observability Gramians for LSS, satisfying generalized Lyapunov equations. Also, we prove that those Gramians encode the reachable and observable sets of an LSS. Based on these Gramians, a balancing-type procedure is proposed enabling to find global projectors $V$ and $W$ to construct a reduced order model. Also, under certain assumptions, the proposed procedure is shown to preserve quadratic stability and to have an error bounds.  However,  since those assumptions are difficult to be checked in the large-scale context, one possible future research axis is to  find whether weaker assumptions exist such that similar results are also valid. Finally, the results are illustrated by some numerical examples. 

\section*{Acknowledgement}
The authors thank the German Research Foundation for funding this work within the CRC/TR 96.
\appendix

\section{Appendix: Proof of Theorem \ref{theo:GramCond}}
In what follows, we only prove that $\Rs= \range \Pg$. The proof that $\Os=\range \Qg$ follows analogously. 

The complete proof of Theorem \ref{theo:GramCond} requires the following propositions. 
\begin{prop}[Theorem 2.2, \cite{dullerud2013course}]\label{prop:LyapLTI}  Let $A \in \R^{n \times n}$ be a Hurwitz matrix and $B \in \R^{n\times m}$. Then, the Lyapunov equation 
	\begin{equation}\label{eq:LyapLTI}
	A\Pg+\Pg A^T +BB^T = 0 
	\end{equation}  has unique symmetric positive semidefinite solution $\Pg$  satisfying
	\[\range{\Pg} = \sum_{l=0}^{\infty}A^l\range B. \] 
\end{prop}

\begin{prop}\label{prop:Lyap2} Let $A \in \R^{n \times n}$ be a Hurwitz matrix and $B_1, \dots, B_M \in \R^{n\times m}$. Then, the Lyapunov equation 
	\begin{equation}\label{eq:LyapLTI2}
	A\Pg+\Pg A^T +\sum_{j=1}^M B_jB_j^T = 0 
	\end{equation} has unique symmetric positive semidefinite $\Pg$ such that
	\[\range \Pg =  \sum_{l=0}^{\infty} \sum_{j=1}^MA^l \range{B_j} \] 
\end{prop}
\begin{proof}
	Let $\Pg_j$ be the unique solution of
	\[ A\Pg_j +\Pg_j A^T +B_jB_j^T = 0. \]
	Then, by linearity, $\Pg = \sum_{j=1}^m \Pg_j$ is the solution of the Lyapunov equation \eqref{eq:LyapLTI2}. Moreover, since $\Pg_j$ is a symmetric positive semidefinite matrix,
	$\range \Pg = \sum_{j=1}^m \range{\Pg_j}$ and the result follows as a consequence of Proposition \ref{prop:LyapLTI}.
\end{proof}

\begin{prop}\label{prop:Lyap3} Let $A \in \R^{n \times n}$ be a Hurwitz matrix and $\Pg_{k-1}$ be a symmetric positive semidefinite matrix. Then, the Lyapunov equation 
	\begin{equation}\label{eq:LyapnovTruncated}
	A\Pg_{k}+\Pg_kA^T + \sum_{j=1}^M D_j\Pg_{k-1}D_j^T = 0 
	\end{equation}has unique symmetric positive semidefinite solution $\Pg_k$ such that
	\[\range{\Pg_k} = \sum_{l=0}^{\infty} \sum_{j=1}^M A^l D_j \range{\Pg_{k-1}}.   \] 
\end{prop}
\begin{proof} 
	Since $\Pg_{k-1}$ is symmetric positive semidefinite, it has a Chosleky  decomposition given by 
	$ \Pg_{k-1} = L_{k-1}L_{k-1}^T$. Then, equation \eqref{eq:LyapnovTruncated} can be rewritten as
	\[
	A\Pg_{k}+\Pg_kA^T + \sum_{j=1}^M D_jL_{k-1}L_{k-1}^TD_j^T = 0 
	\]
	If we rewrite $\tilde B_j = D_jL_{k-1}$, by applying Propostion \ref{prop:Lyap2} and using the fact that $\range{L_{k-1}} =\range{\Pg_{k-1}}$, the result follows.	
\end{proof}

\begin{prop}\label{prop:Lyap4} Let $A \in \R^{n \times n}$ be a Hurwitz matrix. Suppose $\Pg$ is the unique solution of 
	\[A\Pg + \Pg A^T + \ds\sum_{j=1}^{M} \bigg(D_j\Pg D_j^T+  B_jB_j^T \bigg) = 0. \]
	To simplify the notation, let us denote $A = D_{M+1}$. Then, 
	\[ \range{\Pg} =\sum_{k=1}^{\infty}\left(\sum_{\substack{
			i_0, \dots, i_{k} \in \Omega \cup \{M+1\} \\
			j_1, \dots, j_{k}\in \mathbb{N}}} D_{i_{k}}^{j_{k}}\dots D_{i_1}^{j_1}\range{B_{i_0}}\right).  \]
\end{prop}
\begin{proof}
	As stated in Section 2, $\Pg = \sum_{k=1}^{ \infty} \Pg_k$. Hence, since $\Pg_k$ are symmetric positive semidefinite matrices for all $k = 1, 2, \dots$, we must have
	\[ \range{\Pg} = \sum_{k=1}^{ \infty} \range{\Pg_k}. \] 	
	The result follows from Proposition \ref{prop:Lyap2} and by recurrence using Proposition \ref{prop:Lyap3}. 
\end{proof}

Finally, Theorem \ref{theo:GramCond} follows from Proposition~\ref{prop:Lyap4} by  the fact that \[A_j \in \sspan{D_1, \dots, D_M, A}, \forall j\in \Omega\] and 
\[D_j \in \sspan{A_1, \dots, A_M}, \forall j\in \Omega \cup \{M+1\},\]
so that the algebraic conditions given in Theorem \ref{theo:AlgCond} are equivalent to the algebraic condition given in Proposition \ref{prop:Lyap4}.

\bibliography{igorBiblio,mor}


\end{document}